\documentclass{amsart}

\usepackage[utf8]{inputenc}
\usepackage{amsmath,amssymb,amsthm}
\usepackage{mathrsfs}
\usepackage{hyperref}
\usepackage[danish,english]{babel}

\input xy
\xyoption{all}

\theoremstyle{plain}
\newtheorem{thm}{Theorem}
\newtheorem{lemma}{Lemma}
\newtheorem{prop}{Proposition}
\newtheorem{cor}{Corollary}
\newtheorem*{thm:stronglim}{Theorem \ref{stronglim}}
\newtheorem*{thm:dg}{Theorem \ref{thmDG}}
\newtheorem*{prop:hill}{Proposition \ref{perplim}}
\newtheorem*{cor:graded-modules}{Corollary \ref{graded-modules}}
\newtheorem*{thm:concretemodelexample}{Theorem \ref{concretemodelexample}}
\newtheorem*{prop:sheaves}{Proposition \ref{sheaves}}
\newtheorem*{thm:sheaves-noetherian}{Theorem \ref{sheaves-noetherian}}

\theoremstyle{definition}
\newtheorem{defn}{Definition}
\newtheorem{ex}{Example}

\theoremstyle{remark}
\newtheorem{remark}{Remark}

\newtheorem{set}{Setup}

\newcommand{\Pref}[1]{Proposition~\ref{#1}}
\newcommand{\Lref}[1]{Lemma~\ref{#1}}
\newcommand{\Tref}[1]{Theorem~\ref{#1}}
\newcommand{\Coref}[1]{Corollary~\ref{#1}}
\newcommand{\Dref}[1]{Definition~\ref{#1}}
\newcommand{\Rref}[1]{Remark~\ref{#1}}
\newcommand{\Xref}[1]{Example~\ref{#1}}
\newcommand{\Sref}[1]{Setup~\ref{#1}}

\DeclareMathOperator{\Ext}{Ext}
\DeclareMathOperator{\Id}{Id}
\DeclareMathOperator{\Ima}{Im}
\DeclareMathOperator{\FP}{FP}
\DeclareMathOperator{\Ab}{Ab}
\DeclareMathOperator{\Ch}{Ch}
\DeclareMathOperator{\QCoh}{QCoh}
\DeclareMathOperator{\Mod}{Mod}
\newcommand{\dlim}{\varinjlim}

\newcommand{\sla}{\hspace{0.5pt}/\hspace{0.5pt}}
\newcommand{\Z}{\mathbb Z}

\newcommand{\A}{\mathcal A}
\newcommand{\B}{\mathcal B}

\renewcommand{\P}{\mathcal P}
\newcommand{\E}{\mathcal E}
\newcommand{\X}{\mathcal X}
\let\sectionsymbol\S
\renewcommand{\S}{\mathcal S}

\newcommand{\C}{\mathscr C}

\newcommand{\eps}{\varepsilon}
\renewcommand{\phi}{\varphi}
\renewcommand{\subset}{\subseteq}

\newcommand{\fp}[1]{\mathrm{fp}(#1)}

\usepackage[usenames,dvipsnames,svgnames,table]{xcolor}

\begin{document}
\title{Dualizable and semi-flat objects in abstract module categories}
\author{Rune Harder Bak}
\address{Department of Mathematical Sciences, University of Copenhagen, Universitetsparken~5, 2100 Copenhagen {\O}, Denmark} 
\email{bak@math.ku.dk}

\begin{abstract}
In this paper, we define what it means for an object in an abstract module category to be dualizable and we give a homological description of the direct limit closure of the dualizable objects. Our  description recovers existing results of Govorov and Lazard, Oberst and R{\"o}hrl, and Christensen and Holm. When applied to differential graded modules over a differential graded algebra, our description yields that a DG-module is semi-flat if and only if it can be obtained as a direct limit of finitely generated semi-free DG-modules. We obtain similar results for graded modules over graded rings and for quasi-coherent sheaves over nice schemes.
\end{abstract}

\maketitle

\section{Introduction}
In the literature, one can find several results that describe how some kind of ``flat object'' in a suitable category can be obtained as a direct limit of simpler objects. Some examples are:
\begin{enumerate}
\item In 1968 Lazard \cite{lazard68}, and independently Govorov \cite{gov65} proved that over any ring, a module is flat if and only if it is a direct limit of finitely generated projective modules.
\item 
In 1970 Oberst and R{\"o}hrl \cite[Thm 3.2]{oberst70} proved that an additive functor on a small additive category is flat if and only if it is a direct limit of representable functors.
\item 
In 2014 Christensen and Holm \cite{holm14} proved that over any ring, a complex of modules is semi-flat if and only if it is a direct limit of perfect complexes (= bounded complexes of finitely generated projective modules).
\item 
In 1994 Crawley-Boevey \cite{boevey94} proved that over certain schemes, a quasi-co\-he\-rent sheaf is locally flat if and only if it is a direct limit of locally free sheaves of finite rank. In 2014 Brandenburg \cite{brandenburg14} defined another notion of flatness and proved one direction for more general schemes.

\end{enumerate}

In Section~\ref{Abstract-module-categories} we provide a categorical framework that makes it possible to study results and questions like the ones mentioned above. It is this framework that the term ``abstract module categories'' in the title refers to. From a suitably nice (axiomatically described) class $\S$ of objects in such an abstract module category $\C$, we define a notion of semi-flatness (with respect to $\S$). This definition depends only on an abstract tensor product, which is built into the aforementioned framework, and on a certain homological condition. We write $\dlim\S$ for the class of objects in $\C$ that can be obtained as a direct limit of objects from $\S$. Our main result shows that under suitable assumptions, $\dlim\S$ is precisely the class of semi-flat objects:

\begin{thm:stronglim}
Let $\C$ and $\S$ be as in \Sref{setup1} and \Sref{setup2}. In this case, an object in $\C$ is semi-flat if and only if it belongs to $\dlim\S$.
\end{thm:stronglim}

The proof of this theorem is a generalization of the proof of \cite[Thm.~1.1]{holm14}, which in turn is modelled on the proof of \cite[Chap.~I, Thm.~1.2]{lazard68}. A central new ingredient in the proof of \Tref{stronglim} is an application of the generalized Hill Lemma by Stovicek \cite[Thm 2.1]{stovicek13}.

The abstract module categories treated in Section~\ref{Abstract-module-categories} encompass more ``concrete'' module categories such as the category ${}_A\C\sla\C_A$ of left/right modules over a monoid (= ring object) $A$ in a closed symmetric monoidal abelian category $(\C_0,\otimes_1,1,[-,-])$; see
Pareigis~\cite{pareigis86}. In this setting, \Tref{stronglim} takes the form:

\begin{thm:concretemodelexample}
  Let $A$ be a monoid in a closed symmetric monoidal Grothendieck category $(\C_0,\otimes_1,[-,-],1)$ and let ${}_A\C\sla\C_A$ be the category of left/right $A$-modules. Let ${}_A\S$ be (a suitable subset of, e.g.~all) the \emph{dualizable} objects in ${}_A\C$. If $\C_0$ is generated by dualizable objects and $1$ is $\FP_2$, then the direct limit closure of ${}_A\S$ is precisely the class of semi-flat objects in ${}_A\C$.
\end{thm:concretemodelexample}

Dualizable objects in symmetric monoidal categories were defined and studied by Lewis and May in \cite[III\sectionsymbol 1]{equiva} and investigated further by Hovey, Palmieri, and Strickland in \cite{HPS97}; we extend the definition and the theory of such objects to categories of $A$-modules (see Definition \ref{finitedef}).

In the final Section~\ref{Examples}, we specialize our setup even further. For some choices of a closed symmetric monoidal abelian category $\C_0$ and of a monoid $A \in \C_0$, the category of $A$-modules turn out to be a well-known category in which the dualizable and the semi-flat objects admit hands-on descriptions. When applied to differential graded modules over a DGA, to graded modules over a graded ring, and to sheaves over a scheme, Theorem~\ref{concretemodelexample} yields the following results, which all seem to be new.

\begin{thm:dg}
Let $\S$ be the class of finitely generated semi-free/semi-projective diffe\-rential graded modules over a differential graded algebra $A$.
The direct limit closure of $\S$ is precisely the class of semi-flat (or DG-flat) differential graded $A$-modules.
\end{thm:dg}

\begin{cor:graded-modules}
    Over any $\mathbb{Z}$-graded ring, the direct limit closure of the finitely generated projective (or free) graded modules is precisely the class of
    flat graded modules.
\end{cor:graded-modules}

\begin{thm:sheaves-noetherian}
  Let $X$ be a noetherian scheme with the strong resolution property. In the category $\QCoh(X)$, the direct limit closure of the locally free sheaves of finite rank is precisely the class of semi-flat sheaves.
\end{thm:sheaves-noetherian}

In the same vein, it follows that the results (1)--(3), mentioned in the beginning of the Introduction, are also consequences of Theorems~\ref{stronglim} and \ref{concretemodelexample}. 

\section{Preliminaries}

\subsection{Locally finitely presented categories}\label{locpres}

We need some facts about locally finitely presented categories from Breitsprecher \cite{breit79}. 
Let $\C$ be a category. First recall:

\begin{defn}
  \label{defn:generate}
  A collection of objects $\S$ is said to \emph{generate} $\C$ if 
  given different maps $f,g\colon A\to B$ there exists a map $\sigma\colon S\to A$ with $S\in\S$ such that $f\sigma$ and $g\sigma$ are different.
  If $\C$ is abelian, this simply means that if $A\to B$ is non-zero there is some $S\to A$ with $S\in\S$ such that $S\to A\to B$ is non-zero.
\end{defn}

\begin{defn}
  \label{defn:lfp-cat}
  An object $K\in\C$ is called \emph{finitely presented} if $\C(K,-)$ commutes with filtered colimits. Denote by $\fp{\C}$ the collection of all finitely presented objects in $\C$.
  A Grothendieck category is called \emph{locally finitely presented} if it is generated by a small set (as opposed to a class) of finitely presented objects.
\end{defn}

\begin{remark}
  \label{lfp-eqc}
 By \cite[SATZ 1.5]{breit79} a Grothendieck category is locally finitely presented if and only if $\dlim \fp{\C} = \C$, and by \cite[(2.4)]{boevey94} this is equivalent to saying that $\C$ is abelian, $\fp{\C}$ is small, and $\dlim \fp{\C} = \C$.
\end{remark}

\begin{prop} \label{fppresentation}
  Let $\C$ be a Grothendieck category. Then
  \begin{enumerate}

\item \cite[SATZ 1.11]{breit79}   If $\S$ is a set of finitely presented objects generating $\C$, then $N\in\C$ is
finitely presented iff it has a presentation 
$$ 
\xymatrix{
X_0 \ar[r] & X_1 \ar[r] & N \ar[r] & 0
}
$$
where $X_0,X_1$ are finite sums of elements of $\S.$
\item\cite[SATZ 1.9]{breit79} The finitely presented objects are closed under extensions.
  \end{enumerate}
\end{prop}

Next we look at some properties of the class $\dlim\S$ of objects in $\C$ that can be obtained as a direct limit of objects from $\S$.

\begin{lemma}\cite[Lemma p.~1664]{boevey94}\label{limclosed}  
Let $\C$ be a locally finitely presented Grothendieck category, let $M\in\C$ and let $\S$ be a collection of finitely presented objects closed under direct sums.
If any map from a finitely presented object to $M$ factors through some $S\in\S,$
then $M\in\dlim\S.$
In particular $\dlim\S$ is closed under direct limits and direct summands.
\end{lemma}

\begin{remark}
  Notice that the converse is true by definition for any $\S$ and $\C$.
\end{remark}

We will later need the following way of extending the defining isomorphism of an adjunction to the level of Exts.

\begin{lemma}\cite[Lem.~5.1]{HJ}\label{HJ}
  Let $F \colon \mathscr{C} \leftrightarrows \mathscr{D} \,\colon\!G$ be an adjunction of abelian categories, where $F$ is left adjoint of $G$, and let $A \in \mathcal{C}$ be an object. If $G$ is exact and if $F$ leaves every short exact sequence $0 \to A' \to E \to A \to 0$ (ending in $A$) exact, then there is a natural isomorphism $\Ext_\mathscr{D}^1(FA,-) \cong \Ext_\mathscr{C}^1(A,G\mspace{2mu}-)$.
\end{lemma}

\subsection{Cotorsion pairs}
The theory of cotorsion pairs goes back to Salce~\cite{salce79} and has been intensively studied. See for instance G{\"o}bel and Trlifaj \cite{trlifaj06}.
\begin{defn}
Let $\X$ be a class of objects in an abelian category $\C$. We define
\begin{itemize}
\item $\X^\perp=\{Y\in\C\mid \forall X\in\X\colon \Ext_\C^1(X,Y)=0\}$ \vspace*{0.5ex}
\item ${}^\perp\X=\{Y\in\C\mid \forall X\in\X\colon \Ext_\C^1(Y,X)=0\}$
\end{itemize}
\end{defn}

\begin{defn}
  Let  $\A$ and $\B$ be classes of objects in an abelian category $\C.$ 
We say $(\A,\B)$ is a \emph{cotorsion pair}, if
$\A^\perp=\B$ and ${}^\perp \B =\A$. It is \emph{complete} if every $C\in\C$ has 
a presentation
$$0\to B\to A\to C\to 0$$
with $A\in\A$ and $B\in\B$ and a presentation
$$0\to C\to B'\to A'\to 0$$
with  $A'\in\A$ and $B'\in\B$. In this paper, we are only concerned with the first presentation.
\end{defn}
\begin{defn}
  An $\S$-filtration of an object $X$ in a category $\C$ for a class of objects $\S$ is a chain
$$
0=X_0\subset\cdots\subset X_i\subset\cdots\subset X_\alpha=X
$$
of objects in $\C$ such that every $X_{i+1}/X_i$ is in $\S$, and for every limit ordinal 
$\alpha'\leq\alpha$ one has $\dlim_{i<\alpha'} X_i=X_{\alpha'}$.
An object $X$ called $\S$-filtered if it has an $\S$-filtration.
If $\alpha=\omega$ we say the filtration is countable, and if $\alpha<\omega$ that it is
finite. In the latter case we will also say that $X$ is a finite extension of $\S.$
\end{defn}

\begin{prop}\label{completecotorsion}
  If $\S$ is any generating set of objects in a Grothendieck category, then $({}^\perp(\S^\perp),\S^\perp)$
is a complete cotorsion pair, and the objects in ${}^\perp(\S^\perp)$ are precisely the direct summands of $\S$-filtered objects.
\end{prop}
\begin{proof}
  See Saor{\'{\i}}n and {\v{S}}{\v{t}}ov{\'{\i}}{\v{c}}ek \cite[Exa.~2.8 and Cor.~2.15]{stovicek11} (for the last assertion also see {\v{S}}{\v{t}}ov{\'{\i}}{\v{c}}ek \cite[Prop.~1.7]{stovicek13}).
\end{proof}

When $\C$ is locally finitely presented and $\S$ consists of finitely presented objects and is closed under extensions, 
we can in fact realize any $S\in{}^\perp(\S^\perp)$ as a direct limit.
This generalizes the idea that an arbitrary direct sum can be realized as a direct limit of finite sums.
The tool that allows us to generalize this idea is the generalized Hill Lemma.
The full statement is rather technical so we just state here what we need (hence ``weak version''):
\begin{lemma}[Hill Lemma -- weak version]\label{hilllemma}\cite[Thm 2.1]{stovicek13} 
Let $\C$ be a locally finitely presented Grothendieck category, $\S$ be a set of finitely presented objects,
and assume $X$ has an $\S$-filtration. Given any map $f\colon S\to X$ from a finitely presented object,
then $\Ima(f)\subset S'\subset X$ for some finite extension $S'$ of elements of $\S.$
\end{lemma}
We can now prove:
\begin{prop}\label{perplim}
  Let $\S$ be a skeletally small class of finitely presented objects closed under finite extensions in a locally finitely presented Grothendieck category $\C$.
  Then any $\S$-filtered object is a direct limit of objects from $S.$ In particular, ${}^\perp(\S^\perp)\subset\dlim \S$ when $\S$ generates $\C$. 
\end{prop}

\begin{proof}
Let $X$ be an $\S$-filtered object. Since $\C$ is locally finitely presented, $X$ is also the direct limit
of finitely presented objects $X_i$, hence also the direct limit of
its finitely generated subobjects (images of finitely presented objects), but these are majored by $\S$-subobjects by \Lref{hilllemma},
since $\S$ is closed under finite extensions.
The last statement follows from \Pref{completecotorsion} and \Lref{limclosed}.
\end{proof}

\section{Abstract module categories}
\label{Abstract-module-categories}
The aim in this section is to desribe the direct limit closure of $\S$ in the following setup:
\begin{set}\label{setup1}
  Let $\C_L$, $\C_0$ and $\C_R$ be Grothendieck categories,
let $\S_L\subset\C_L$ and $\S_R\subset\C_R$ be generating sets 
closed under extensions, and let $1\in\C_0$ be finitely presented. Assume that we have a right continuous bifunctor
(i.e. it preserves direct limits in each variable)
$$-\otimes -\colon \C_R\times\C_L\to\C_0$$
and a natural duality
\begin{align*}
 (-)^*&\colon \S_L\to\S_R
\end{align*}
such that for any $S\in\S_L$ we have natural isomorphisms (also natural in $S$):
\begin{align*}
\C_0(1,S^*\otimes-)&\cong\C_L(S,-) \quad \text{and} \\
\C_0(1,-\otimes S)&\cong\C_R(S^*,-)
\end{align*}
which is then analogously true for any $S\in\S_R$ by the duality between $\S_L$ and $\S_R$. For simplicity we will often write $\C$ for either $\C_L$ or $\C_R$ and $\S$ for either $\S_L$ and~$\S_R$ (see for example \Tref{stronglim}). Hopefully this should not cause any confusion.
\end{set}

\begin{remark}\label{finiteisfp}
  Note that in \Sref{setup1} any $S\in\S$ is finitely presented because $1$ is finitely presented and $\otimes$ is right continuous, so $\C_L$ and $\C_R$ are necessarily locally finitely presented.
  When there are notational differences we will work with $\C_L$ though everything could be done for $\C_R$ instead. 
\end{remark}

\begin{ex}
\label{examples-of-Setup1}
Some specific examples of \Sref{setup1} to have in mind are:
\begin{enumerate}
\setlength{\itemsep}{0.5ex}

\item \label{item-Mod} $A$ is a ring, $\C_L$\sla $\C_R$ is the category $A\text{-Mod}$\sla $\text{Mod-}A$ of left\sla right $A$-modules, \mbox{$\C_0=\Ab$} is the category of abelian groups, $1$ is $\mathbb{Z}$, $\otimes = \otimes_A$ is the ordinary tensor product of modules, $\S_L$\sla $\S_R$ is the category of finitely generated projective left\sla right $A$-modules, and $(-)^*$ is the functor $\mathrm{Hom}_A(-,{}_AA_A)$.

\item \label{item-GrMod} $A$ is a graded ring, $\C_L$\sla $\C_R$ is the category $A\text{-GrMod}$\sla $\text{GrMod-}A$ of left\sla right graded $A$-modules, $\C_0$ is $\mathbb{Z}\text{-GrMod}$, $1$ is $\mathbb{Z}$, \mbox{$\otimes = \otimes_A$} is the ususal tensor product of graded modules, $\S_L$\sla $\S_R$ is the category of finitely generated free graded left\sla right $A$-modules (that is, finite direct sums of shifts of $A$), and $(-)^*$ is the functor $\mathrm{Hom}_A(-,{}_AA_A)$.

\item \label{item-Ch} $A$ is a ring, $\C_L$\sla $\C_R$ is the category $\text{Ch}(A\text{-Mod})$\sla $\text{Ch}(\text{Mod-}A)$ of chain complexes of left\sla right $A$-modules, $\C_0$ is $\text{Ch}(\Ab)$, $1$ is $\mathbb{Z}$ (viewed as a complex concentrated in degree zero), \mbox{$\otimes = \otimes_A$} is the total tensor product of chain complexes, $\S_L$\sla $\S_R$ is the category of bounded chain complexes of finitely generated projective left\sla right $A$-modules (these are often called \emph{perfect complexes}), and $(-)^*$ is the functor $\mathrm{Hom}_A(-,{}_AA_A)$. 

\item \label{item-DGMod} $A$ is a DGA, $\C_L$\sla $\C_R$ is the category $A\text{-DGMod}$\sla $\text{DGMod-}A$ of left\sla right DG $A$-modules, $\C_0$ is $\text{Ch}(\Ab)$, $1$ is $\mathbb{Z}$ (as in (3)), \mbox{$\otimes = \otimes_A$} is the ususal tensor product of DG-modules, $\S_L$\sla $\S_R$ is the category of finitely generated semi-free left\sla right DG $A$-modules (that is, finite extensions of shifts of $A$), and $(-)^*$ is the functor $\mathrm{Hom}_A(-,{}_AA_A)$.

\item \label{item-monoidal} Let $(\C_0,\otimes_1,1,[-,-])$ be any closed symmetric monoidal abelian category where $1$ is finitely presented.
Then one can take $\C_L = \C_0 = \C_R$ and $\otimes=\otimes_1$. Moreover, $\S_L = \S_R$ could~be the subcategory of \emph{dualizable} objects in $\C_0$ (see \ref{subsec:dualizable}) and $(-)^*=[-,1]$. 

\end{enumerate}

These examples are all special cases of the ``concrete module categories'' studied in Section~\ref{Concrete-module-categories}, and further in Section~\ref{Examples}. A special case of (4) is where $\C_0=\QCoh(X)$ is the category of quasi-coherent sheaves on a sufficiently nice scheme $X$ and where $\S_L = \S_R$ is the category of locally free sheaves of finite rank; see \ref{subsection:QCoh} for details.

\begin{enumerate}
\setcounter{enumi}{5}
\setlength{\itemsep}{0.5ex}

\item \label{item-func} $\X$ is an additive category, $\C_L$\sla $\C_R$ is the category $[\X,\Ab]$\sla $[\X^\mathrm{op},\Ab]$ of~co\-va\-ri\-ant\sla contravariant additive functors from $\X$ to $\Ab$, $\C_0$ is $\Ab$, $1$ is $\mathbb{Z}$, \mbox{$\otimes = \otimes_\X$} is the  tensor product from Oberst and R{\"o}hrl \cite{oberst70}, $\S_L$\sla $\S_R$ is the category of representable co\-va\-ri\-ant\sla contravariant functors, and 
the functor $(-)^*$ maps $\X(x,-)$ to $\X(-,x)$ and vice versa ($x \in \X$). See \ref{subsection:additive-functors} for details.

\end{enumerate}
\end{ex}

Recall that to simplify notation we often write $\C$ for either $\C_L$ or $\C_R$ and $\S$ for either $\S_L$ and $\S_R$ (see \Sref{setup1}). In order to describe $\dlim\S$, we define from $\S$ three new classes of objects in $\C$.

\begin{defn}\label{projacycdef}
Let $\C$ and $\S$ be as in \Sref{setup1}. Let $(\P,\E)$ be the cotorsion pair in $\C$ generated by $\S$. By Proposition~\ref{completecotorsion} this cotorsion pair is complete as $\S$ is a set.

Objects in $\P$ are called \emph{semi-projective} and objects in $\E$ are called \emph{acyclic} (with respect to $\S$). An object $M\in\C_L$ is called \emph{(tensor-)flat} if the functor \mbox{$-\otimes M$} is exact.
A functor $F\colon \C\to\C_0$ \emph{preserves acyclicity} if $F(\E)\subset 1^\perp$. Finally we say that an object $M\in\C_L$ is \emph{semi-flat} if $M$ is flat and \mbox{$-\otimes M$} preserves acyclicity.
\end{defn}

When necessary we shall use the more elaborate notation $(\P_L,\E_L)$ for the cotorsion pair in $\C_L$ generated by $\S_L$ and similarly for $(\P_R,\E_R)$.

\begin{ex}\label{flatex}
We immediately see that if $1 \in \C_0$ is projective, then semi-flat is the same as flat.
This is for instance the case in $A\text{-Mod}$, $A\text{-GrMod}$ and $[\X,\Ab]$ (see~(\ref{item-Mod}), (\ref{item-GrMod}), and (\ref{item-func}) in Example \ref{examples-of-Setup1}), where every object is acyclic, and semi-projective is the same as projective. In $\text{Ch}(A) = \text{Ch}(A\text{-Mod})$ and in $A\text{-DGMod}$ (see (\ref{item-Ch}) and (\ref{item-DGMod}) in Example~\ref{examples-of-Setup1}) this is not the case, and the notions acyclic, semi-projective and semi-flat agree with the usual ones found in e.g.~\cite{avramov}. More on this and other examples after the main theorem. 
\end{ex}

We are now ready for the main lemma.
The proof is modelled on \cite[Thm 1.1]{holm14} which is modelled on \cite[Lem 1.1]{lazard68}.
We try to use the same notation.
\begin{lemma}\label{weaklim}
With the notation of \Sref{setup1},
let $M\in\C_L$ be an object such that $-\otimes M$ is left exact  and $\C_0(1,\phi\otimes M)$ is epi whenever $\phi$ is epi in $\C_R$ with $\ker\phi\in\E_R$. Then $M\in\dlim\S_L.$
\end{lemma}

\begin{proof}
By \Lref{limclosed} we need to fill in the dashed part of the following diagram   
$$\xymatrix{
P\ar[r]^u \ar@{-->}[dr]_v & M \\
& L\ar@{-->}[u]_w
}
$$
for some $L\in\S_L$, where $u$ is given with $P$ finitely presented. So let $u$ be given.

By \Pref{fppresentation}, $P$ has a presentation
$$\xymatrix{
L_1\ar[r]^f&L_0\ar[r]^g &P\ar[r]&0
}
$$
with $L_1,L_0\in\S_L$. We have an exact sequence
$$\xymatrix{
0\ar[r]&K\ar[r]^k&L_0^*\ar[r]^{f^*}&L_1^*
},
$$
which, since $-\otimes M$ and $\C_0(1,-)$ are left exact, gives an exact sequence
$$\xymatrix{
0\ar[r]&\C_0(1,K\otimes M)\ar[r]^{k_*}&\C_L(L_0,M)\ar[r]^{f_*}&\C_L(L_1,M)
}
$$
where we have used $\C_0(1,L_j^*\otimes M)\cong \C_L(L_j,M)$ for $j=0,1$.

By completeness of the cotortion pair $(\P_R,\E_R)$, the object $K$ has a presentation
$$\xymatrix{
0\ar[r]&E\ar[r]&L'\ar[r]^\phi &K\ar[r]&0
}
$$
with $L'\in \P_R$ and $E\in\E_R.$
By assumption, $\phi_*=\C_0(1,\phi\otimes M)$ is epi, so we get an exact sequence
$$
\xymatrix{
\C_0(1,L'\otimes M) \ar[r]^-{k_*\phi_*} & \C_L(L_0,M) \ar[r]^{f_*} & \C_L(L_1,M).
}
$$
Now since $f_*(ug)=ugf=0$, we have some $w'\colon1\to L'\otimes M$ such that $(k\phi)_*(w')=ug.$
By \Pref{perplim} we can realize $L'$ as a direct limit $\dlim L^*_i,$ with $L_i\in\S_L.$
This means that we have 
$$L'\otimes M\cong(\dlim L^*_i)\otimes M\cong\dlim(L^*_i\otimes M),$$
as $\otimes$ is right continuous.
Since $1$ is finitely presented, $w'$ factors as
$$\xymatrix{1\ar[r]^-{w}&L^*\otimes M\ar[r]^{\iota\otimes M}\ar[r]&L'\otimes M}$$
for some $L \in \S_L$ and $w\in\C_L(L,M)\cong\C_0(1,L^*\otimes M).$
By the assumed duality between $\S_L$ and $\S_R$ there exists $v'\colon L_0\to L$ such that $v'^*=k\phi\iota$. We now have the commutative diagram
$$
\xymatrix{
\C_L(L,M)\ar[d]^{\iota_*}\ar[dr]^{v'_*} \\
\C_0(1,L'\otimes M)\ar[r]^{k_*\phi_*}&\C_L(L_0,M)\ar[r]^{f_*} &\C_L(L_1,M)
}
$$
where $wv'=v'_*(w)=k_*(\phi_*(w'))=ug.$
This gives us the commutative diagram
$$\xymatrix{
L_1\ar[r]^f\ar[drr]_0&L_0\ar[r]^g\ar[dr]^{v'} &P\ar@{-->}[d]^v\ar[dr]^u\ar[r] & 0\\
&& L\ar[r]^w &M
}$$
where $v'f=0$ since $f^*v'^*=f^*k\phi\iota=0\phi\iota=0$. Thus $v'$ factors through $g$ by some $v$ as $g$ is the cokernel of $f$. It remains to note that $wv=u$, as desired.
\end{proof}

\begin{remark}
  The main difference between this proof and the proof in \cite{holm14} is that
  all the relevant identities have been formalized instead of based on calculations with elements,
  in particular, the use of the generalized Hill Lemma instead of element considerations
  to find the right $\S$-subobject of a semi-projective object.
\end{remark}

Lemma~\ref{weaklim} will allow us to prove that every semi-flat object belongs to the direct limit closure of $\S$ (see Theorem~\ref{stronglim} below). For the converse statement, we need the following setup.

\begin{set}\label{setup2}
  With the notation of \Sref{setup1} assume further that
  $1$ is $\FP_2$, i.e. $\Ext_{\C_0}(1,-)$ respects direct limits, and that
  for any $S\in\S_L$ we have that $-\otimes S$ is exact and there are natural isomorphisms:
  \begin{align*}
    \Ext_{\C_0}(1,S^*\otimes-)&\cong\Ext_{\C_L}(S,-) \quad \text{and}\\
    \Ext_{\C_0}(1,-\otimes S)&\cong\Ext_{\C_R}(S^*,-)\,.
\end{align*}
  By the duality between $\S_L$ and $\S_R$, similar conditions hold for $S\in\S_R$. (Note that the isomorphisms above are the ``Ext versions'' of the isomorphisms from \Sref{setup1}.)
\end{set}

As in \Rref{finiteisfp} one sees that in the setting of \Sref{setup2} every $S \in \S$  is $\FP_2$.

We can now link the direct limit closure to semi-flatness (from Definition~\ref{projacycdef}).

\begin{thm}\label{stronglim}
Let $\C$ and $\S$ be as in \Sref{setup1}.
If $M\in{}\C$ is semi-flat, then $M\in\dlim\S.$
Conversely, if $\C$ and $\S$ satisfy the conditions of \Sref{setup2}, then every $M\in\dlim\S$ is semi-flat.
\end{thm}

\begin{proof}
  To use \Lref{weaklim}, we just need to see, that if $M\in\C_L$ is semi-flat, then $\C_0(1,\phi\otimes_A M)$ is epi whenever 
$\phi$ is epi and $\ker\phi$ is acyclic.
This is clear, since if
$$
\xymatrix{0\ar[r]&E\ar[r]&A\ar[r]^\phi&B\ar[r]&0}
$$
is exact and $E$ is acyclic, then
$$
\xymatrix{0\ar[r]&E\otimes M\ar[r]&A\otimes M\ar[r]^{\phi\otimes M}&B\otimes M\ar[r]&0}
$$
is exact  and $\Ext_{\C_0}(1,E\otimes M)=0.$
But this implies that $\C_0(1,\phi\otimes M)$ is epi. 

For the other direction we show that every $S \in \S_L$ is semi-flat and that the class of semi-flat objects in $\C_L$ is closed under direct limits. First observe that if $E\in \E_R=\S_R^\perp$ and $S\in\S_L$ then 
$$\Ext_{\C_0}(1,E\otimes S)\cong\Ext_{\C_R}(S^*,E)=0\,,$$ so $-\otimes S$ preserves acyclicity, and
since $-\otimes S$ is assumed to be exact, $S$ is semi-flat.
Now if $M_i\in\C_L$ is a direct system of semi-flat objects and $M=\dlim M_i,$
then $-\otimes M$ is exact as $\otimes$ is right continuous and $\dlim(-)$ is exact.
It also preserves acyclicity, as 
$$\Ext_{\C_0}(1,E\otimes\dlim M_i)\cong\Ext_{\C_0}(1,\dlim (E\otimes M_i))\cong\dlim\Ext_{\C_0}(1,E\otimes M_i)=0$$
as $\otimes$ is right continuous,  $\Ext_{\C_0}(1,-)$ respects direct limits and $\Ext_{\C_0}(1,E\otimes M_i)=0.$
Hence $M$ is semi-flat.
\end{proof}

\begin{cor}\label{stronglimcor}
  Let $\C$ and $\S$ be as in \Sref{setup1} and assume further that that $1 \in \C_0$ is projective and that every $S\in\S$ is (tensor-)flat. Then the direct limit closure of $\S$ is the class of (tensor-)flat objects in $\C$.
\end{cor}

\begin{proof}
  As in \Xref{flatex} semi-flat just means flat if $1$ is projective, so by \Tref{stronglim}
  every flat object in $\C$ is in the direct limit closure of $\S$.
  On the other hand any $S\in\dlim\S$ is flat as this is preserved by direct limits as in the proof of \Tref{stronglim}.
\end{proof}

As mentioned in the Introduction, we will now see how \Tref{stronglim} recovers Govorov and Lazard's original theorem for modules, the theorem by Christensen and Holm
for complexes of modules, the theorem by Oberst and R{\"o}hrl for functor categories, and how it gives new results for graded modules, differential graded modules, and quasi-coherent sheaves.

Most of these examples are built of categories of left/right objects for some monoid in a symmetric monoidal category.
So in the next section we will explain this construction with a new definition of \emph{dualizable} objects in such categories
and show in what cases they satisfy \Sref{setup1} and \ref{setup2}.
Then we will go in depth with the more concrete examples, calculating the different classes of objects.

\section{Concrete module categories}
\label{Concrete-module-categories}
\begin{set}\label{setup3}
The details of this setup can be found in Pareigis~\cite{pareigis86}.
Consider any closed symmetric monoidal abelian category $\C_0=(\C_0,\otimes_1,[-,-],1)$.
A \emph{monoid} (or a \emph{ring object}) in $\C_0$ is an object, $A,$ together with an associative multiplication \mbox{$A\otimes_1 A\to A$} with a unit \mbox{$1\to A$}.
We can then consider the category ${}_A\C$ of \emph{left $A$-modules} whose objects are
objects $X\in\C_0$ equipped with a left $A$-multiplication \mbox{$A\otimes_1 X\to X$} respecting the multiplication of $A$ on the left and the unit.
The morphisms are morphims in $\C_0$ respecting this left $A$-multiplication.
We can also consider the category $\C_A$ of \emph{right $A$-modules}
and the category ${}_A\C_A$ of \emph{$(A,A)$-bimodules}, that is, simultaneously left and right $A$-modules with compatible actions.

We can then construct a functor $\otimes_A\colon\C_A\times {}_A\C\to\C_0$ as a coequalizer:
$$
\xymatrix{
Y\otimes_1 A \otimes_1 X \ar@<-.5ex>[r] \ar@<.5ex>[r]& Y\otimes_1 X\ar[r]& Y\otimes_A X
}.
$$
And we get induced functors ${}_A\C_A\times {}_A\C\to{}_A\C$ and  $\C_A\times {}_A\C_A\to\C_A$ with $A\otimes_A X\cong X$ in ${}_A\C$ and $Y\otimes_A A\cong Y$ in $\C_A$.

We can also construct  ${}_A[-,-]\colon{}_A\C\times{}_A\C\to\C_0$ as an equalizer 
$$
\xymatrix{
{}_A[X,X']\ar[r]&[X,X']\ar@<-.5ex>[r] \ar@<.5ex>[r]&[A\otimes_1 X,X']
}
$$
and similarly for $[-,-]_A\colon \C_A\times\C_A\to\C_0$. Again we get induced functors ${}_A[-,-]\colon{}_A\C\times{}_A\C_A\to\C_A$ and $[-,-]_A\colon \C_A\times{}_A\C_A\to{}_A\C$. 

There are natural isomorphisms:
\begin{align*}
{}_A\C(X\otimes_1 Z,X')&\cong\C_0(Z,{}_A[X,X'])\,, \\ 
\C_A(Z\otimes_1 Y,Y')&\cong\C_0(Z,[Y,Y']_A)\,, \text{ and} \\ 
{}_A[X\otimes_1 Z,X']&\cong [Z,{}_A[X,X']]\,. 
\end{align*}
That is, \mbox{$X\otimes_1-$} and \mbox{${}_A[X,-]$} (as well as \mbox{$-\otimes_1 Y$} and \mbox{$[Y,-]_A)$} are adjoints. Similarly, $-\otimes_A X$ and $[X,-]$ (as well as $Y\otimes_A -$ and $[Y,-]$) are adjoints.
We denote the unit and the counit of the adjunctions by $\eta$ and $\eps.$
As $A\in{}_A\C_A,$ we can define functors $(-)^*={}_A[-,A]\colon{}_A\C\to\C_A$ and  $(-)^*=[-,A]_A\colon\C_A\to{}_A\C$
with $A^*\cong A$, where on one side, $A$ is regarded as a left $A$-module, and on the other side, $A$ is regarded as a right $A$-module. Also notice that ${}_1\C\cong\C_0\cong\C_1.$
Again all the details are in \cite{pareigis86}.

In accordance with our convention from \Sref{setup1}, we often write $\C$ for either ${}_A\C$ or $\C_A$.
\end{set}

The forgetful functor from ${}_A\C\to\C_0$ creates limits, colimits and isomorphisms \cite[2.4]{pareigis86} and thus we get:

\begin{prop}\label{acisab}
  If $\,\C_0$ is Grothendieck generated by a collection $\{X\}$ of (finitely~pre\-sen\-ted) objects, then ${}_A\C$ is Grothendieck generated by the collection $\{A\otimes_1 X\}$ of~(fi\-nite\-ly presented) objects.
\end{prop}

\begin{proof}
 We only prove the assertion about the generators (see Definition~\ref{defn:generate}). Assume that $\C_0$ is generated by $\{X\}$. Let \mbox{$Y \to Y'$} be a non-zero morphism in ${}_A\C$. Then \mbox{$Y \to Y'$} is also non-zero in $\C_0$, so we can find some $X$ in the collection $\{X\}$ and a morphism \mbox{$f \colon X \to Y$} such that \mbox{$X \to Y \to Y'$} is non-zero in $\C_0$. Now the morphism \mbox{$X \to A \otimes_1 X \to A \otimes_1 Y \to Y \to Y'$} is non-zero as \mbox{$X \to A \otimes_1 X \to A \otimes_1 Y \to Y$} is equal to $f$, and hence \mbox{$A \otimes_1 X \to A \otimes_1 Y \to Y \to Y'$} must be non-zero as well. Thus the collection  $\{A\otimes_1 X\}$ generates  ${}_A\C$. If $X$ is finitely presented in $\C_0$, then $A\otimes_1 X$ is finitely presneted in ${}_A\C$ since ${}_A\C(A\otimes_1 X,-) \cong \C_0(X,{}_A[A,-])$ and the forgetful functor ${}_A[A,-] \colon {}_A\C\to\C_0$ preserves colimits.
\end{proof}

\subsection{Dualizable objects}
\label{subsec:dualizable}

In \cite[III\sectionsymbol 1]{equiva} Lewis and May define \emph{finite} objects in a closed symmetric monoidal category.
Such objects are called \emph{(strongly) dualizable} in
Hovey, Palmieri, and Strickland \cite{HPS97}.
We extend this notion to categories of left/right modules over a monoid in a closed symmetric monoidal category by the following definition.

First, $\varepsilon$ (introduced above) induces a map
\begin{displaymath}
  \C_0(Z,X^* \otimes_A X') \longrightarrow {}_A\C(X \otimes_1 Z,X'),
\end{displaymath}
for any $X,X' \in {}_A\C$ and $Z \in \C_0$, by $X \otimes_1 Z \longrightarrow X \otimes_1 X^* \otimes_A X' \xrightarrow{\varepsilon \otimes_A X'} X'$.

Next,  $\varepsilon$ induces a morphism
\begin{displaymath}
  \nu \colon {}_A[X,Z] \otimes_A X' \longrightarrow {}_A[X,Z\otimes_A X'],
\end{displaymath}
for any $X,X' \in {}_A\C$ and $Z \in {}_A\C_A$, 
by the adjoint of $X \otimes_1 {}_A[X,Z] \otimes_A X' \xrightarrow{\varepsilon \otimes_A X'} Z \otimes_A X'$. We can now give the following:

\begin{defn}\label{finitedef}
  An object $X\in{}_A\C$ is said to be \emph{dualizable} if 
  there exists a morphism \smash{$\eta' \colon 1 \to X^* \otimes_A X$} in $\C_0$ such that the following diagram commutes:
    \begin{displaymath}
      \xymatrix{
        1 \ar[d]_-{\eta} \ar[r]^-{\smash{\eta'}} & X^* \otimes_A X \ar[dl]^-{\nu}
        \\
        {}_A[X,X] & {}
      }
    \end{displaymath}
Similarly, one defines what it means for an object in $\C_A$ to be dualizable.
\end{defn}

Note that $A \in {}_A\C$ and $A \in \C_A$ are always dualizable.

Many equivalent descriptions of dualizable objects can be given, and we give several in the next lemma.

\begin{lemma}\label{finitedeflem}\label{fdl}
  For $X\in{}_A\C$ the following conditions are equivalent:
  \begin{enumerate}
    \setlength{\itemsep}{2pt}
    \item \label{one} There exists a morphism $\eta' \colon 1 \to X^* \otimes_A X$ in $\C_0$ making the following diagram commute:
\begin{displaymath}
  \xymatrix@!=1.2pc{
  X \ar[dr]_-{=} 
  \ar[rr]^-{X \otimes_1 \eta'} & & 
  X \otimes_1 X^* \otimes_A X
  \ar[dl]^-{\varepsilon \otimes_A X}
  \\ 
  {} & X & {} 
  }
\end{displaymath}    

    \item \label{two} $\C_0(1, X^* \otimes_A X) \xrightarrow{\smash{\cong}} \hspace*{-1.5pt}{}_A\C(X,X)$ induced by $\varepsilon$. 
  
    \item \label{three} $\C_0(1, X^* \otimes_A X') \xrightarrow{\smash{\cong}} \hspace*{-1.5pt}{}_A\C(X,X')$ induced by $\varepsilon$ for all $X' \in {}_A\C$. 
    
    \item \label{four} $\C_0(Z, X^* \otimes_A X') \xrightarrow{\smash{\cong}} \hspace*{-1.5pt}{}_A\C(X \otimes_1 Z,X')$ induced by $\varepsilon$ for all $X' \in \hspace*{-1.5pt}{}_A\C$ and~\mbox{$Z \in \C_0$}. 
      
    \item \label{five} \mbox{$\C_0(Z, Y \otimes_A X') \xrightarrow{\smash{\cong}} \hspace*{-1.5pt}{}_A\C(X \otimes_1 Z,X')$} for some $Y \in \C_A$ and all \mbox{$X' \in {}_A\C, Z \in \C_0$}. 

    \item \label{six} $X$ is dualizable.
    
    \item \label{seven} $\nu \colon X^* \otimes_A X \xrightarrow{\smash{\cong}} {}_A[X,X]$.

    \item \label{eight} $\nu \colon {}_A[X,Z] \otimes_A X' \xrightarrow{\smash{\cong}} {}_A[X,Z\otimes_A X']$ for all $X' \in {}_A\C$ and $Z \in {}_A\C_A$.

    \item \label{nine} $\nu \colon {}_A[X',Z] \otimes_A X \xrightarrow{\smash{\cong}} {}_A[X',Z\otimes_A X]$ for all $X' \in {}_A\C$ and $Z \in {}_A\C_A$.
    
  \end{enumerate}
\end{lemma}                     

\begin{proof}
  Clearly \eqref{four} $\Rightarrow$ \eqref{three} $\Rightarrow$ \eqref{two} $\Rightarrow$ \eqref{one}, and \eqref{one} $\Rightarrow$ \eqref{four} as $\eta'$ from \eqref{one} induces a map ${}_A\C(X\otimes_1Z,X') \to \C_0(Z,X^* \otimes_AX')$ by $Z \to X^* \otimes_A X \otimes_1 Z \to X^* \otimes_A X'$, and the diagram from \eqref{one} precisely says that it is an inverse to the map induced by $\varepsilon$. Clearly, either of the conditions \eqref{eight} and \eqref{nine} imply \eqref{seven}, and \eqref{seven} $\Rightarrow$ \eqref{six}. The implications \eqref{six} $\Rightarrow$ \eqref{eight} and \eqref{six} $\Rightarrow$ \eqref{nine} can be proved as in \cite[III Prop.~1.3(ii)]{equiva}. We also have \eqref{one} $\Leftrightarrow$ \eqref{six} as the diagrams in question are adjoint, so we are left with noting that \eqref{four} $\Rightarrow$ \eqref{five} (trivial) and that \eqref{five} $\Rightarrow$ \eqref{six} can be proved as in \cite[III Thm.~1.6]{equiva}.
\end{proof}

\begin{remark}\label{finiterem}
  We notice that \Lref{finitedeflem}~\eqref{five} makes no mention of the functor $[-,-]$ and thus this condition can be used to define dualizable objects in, for example, symmetric monoidal categories that are not closed. In this case, $Y$ is a ``dual''~of~$X$.
  We chose a definition with a fixed dual object, $X^*={}_A[X,A]$, because this emphasizes the canonical and thereby functorial choice of a dual object.
\end{remark}

Next we show three lemmas about closure properties for the class of dualizable objects.

\begin{lemma}\label{biduality}
$(-)^*$ induces a duality between the categories of dualizable objects~in~${}_A\C$ and dualizable objects in $\C_A$.
In particular, if $X$ is dualizable, then so is $X^*$~and~the adjoint of $\eps$ gives an isomorphism \smash{$X \stackrel{\cong}{\longrightarrow} X^{**}$}.
\end{lemma}
\begin{proof}
  As in \cite[Prop. 1.3(i)]{equiva}.
\end{proof}
\begin{lemma}\label{finiteext}
   Dualizable objects are closed under extensions and direct summands.
\end{lemma}
\begin{proof}
The closure under direct summands follows directly from \Lref{finitedeflem}~\eqref{three}.

So assume that
$$
\xymatrix{
0 \ar[r] & X_1 \ar[r] & X_2\ar[r] & X_3\ar[r] & 0
} $$
is exact and $X_1, X_3$ are dualizable (in ${}_A\C$).
Then we have the following commutative diagram in $\C_0$ with exact rows
$$\xymatrix{
 & X_2^*\otimes_AX_1 \ar[r]\ar[d]_{\simeq} & X_2^*\otimes_AX_2\ar[r]\ar[d] & X_2^*\otimes_AX_3\ar[r]\ar[d]_{\simeq} & 0 \\
0\ar[r] & {}_A[X_2,X_1] \ar[r] & {}_A[X_2,X_2]\ar[r] & {}_A[X_2,X_3]\,, &
}$$
where the outer vertical morphisms are isomorphisms by \Lref{finitedeflem}~\eqref{nine},
so the middle morphism is an isomorphism by the snake lemma. Hence $X_2$ is dualizable by \Lref{finitedeflem}~\eqref{seven}.
\end{proof}

\begin{lemma}\label{finiteclosedundertensor}
If $S$ is dualizable in $\C_0$, then $$(X\otimes_1 S)^*\cong S^*\otimes_1 X^*$$ for any $X\in{}_A\C$.
If $X\in{}_A\C$ is dualizable then so is $X\otimes_1 S\in{}_A\C$.
In particular, $A\otimes_1 S\in{}_A\C$ and $(A\otimes_1 S)^*\cong S^*\otimes_1 A\in\C_A$ are dualizable if $S\in\C_0$ is dualizable. 
\end{lemma}

\begin{proof}
  If $S \in \C_0$ is dualizable, then we have 
  $$(X\otimes_1 S)^*={}_A[X\otimes_1 S,A]\cong [S,{}_A[X,A]]\cong[S,1\otimes_1 X^*]\cong[S,1]\otimes_1 X^* = S^*\otimes_1 X^*.$$
When $X$ is also dualizable we have 
$$
\C_0(1,(X\otimes_1 S)^*\otimes_A-)\cong\C_0(1,S^*\otimes_1 X^* \otimes_A -)\cong\C_0(S,X^*\otimes_A -)
\cong{}_A\C(X\otimes_1 S,-)
$$
on ${}_A\C$, and hence $X\otimes_1 S$ is dualizable in ${}_A\C$ by \Lref{finitedeflem}~\eqref{three}.
\end{proof}

We now have a large supply of categories satisfying \Sref{setup1} and \ref{setup2}

\begin{thm}\label{concretemodelexample} 
  Let $A$ be a monoid in a closed symmetric monoidal Grothendieck category $(\C_0,\otimes_1,[-,-],1)$ where $1$ is finitely presented.
  Assume that $\C_0$ is generated by a set $\S$ of dualizable objects such that $\S^*$ also generates $\C_0$ (e.g.~if $\S=\S^*$). Assume furthermore that ${}_A\S$ is a collection of dualizable objects in ${}_A\C$ which is closed under extensions and contains \mbox{$A\otimes_1\S$} (e.g.~${}_A\S$ could be the collection of all dualizable objects in ${}_A\C$; see Lemmas \ref{finiteext} and \ref{finiteclosedundertensor}). 
  
  \begin{enumerate}

  \item Under the assumptions above, the data $\C_L\mspace{-2mu}:=\mspace{-2mu}{}_A\C$, $\C_R\mspace{-2mu}:=\mspace{-2mu}\C_A$, $\otimes\mspace{-2mu}:=\mspace{-2mu}\otimes_A$, $(-)^*\mspace{-2mu}:=\mspace{-2mu}{}_A[-,A]$, $\S_L\mspace{-2mu}:=\mspace{-2mu}{}_A\S$ and \mbox{$\S_R\mspace{-2mu}:=\mspace{-2mu}({}_A\S)^*$} satisfy \Sref{setup1}. 
  
  In particular, \Tref{stronglim} yields that every semi-flat object in ${}_A\C$, respectively, in $\C_A$, belongs to $\dlim {}_A\S$, respectively, to $\dlim \S_A$.
  
  \item If, in addition, $1$ is $\FP_2$, then \Sref{setup2} holds as well.
  
  In particular, \Tref{stronglim} yields that the class of semi-flat objects in ${}_A\C$, respectively, in $\C_A$, is precisely $\dlim {}_A\S$, respectively, $\dlim \S_A$.
  
  \item If, in addition, $1$ is projective, then $\dlim {}_A\S$, respectively, $\dlim \S_A$, is precisely the (tensor-)flat objects in ${}_A\C$, respectively, in $\C_A$.
  \end{enumerate}
\end{thm}

\begin{proof}  
  (1): First note that since $1 \in \C_0$ is finitely presented, so is every dualizable object. Indeed, for e.g.~$S \in {}_A\C$ one has ${}_A\C(S,-) \cong \C_0(1,S^* \otimes_A-)$;~cf.~Remark~\ref{finiteisfp}. \Pref{acisab} shows that ${}_A\C$ is Grothendieck generated by the set~\mbox{$A\otimes_1\S \subseteq {}_A\S$}. The objects in the set $A\otimes_1\S$ are dualizable, cf.~\Lref{finiteclosedundertensor}, and hence finitely presented by the observation above. Consequently, ${}_A\C$ is a locally finitely presented Grothendieck category, and $\fp{{}_A\C}$ is small by Remark~\ref{lfp-eqc}; hence ${}_A\S \subseteq \fp{{}_A\C}$ is small. Similarly, $\C_A$ is  Grothendieck generated by the set $\S^*\otimes_1A=(A\otimes_1\S)^* \subseteq ({}_A\S)^*$; see Lemma~\ref{finiteclosedundertensor}. And as ${}_A\S$ is small, so is $({}_A\S)^*$. By \Lref{biduality} the class $({}_A\S)^*$ consists of dualizable objects and $(-)^*$ yields a duality between ${}_A\S$ and $({}_A\S)^*$. Since ${}_A\S$ is closed under extensions, the same is true for $({}_A\S)^*$ (by the duality). The natural isomorphisms in \Sref{setup1} hold by \Lref{finitedeflem}~\eqref{three}. It remains to note that $\otimes_A$ is a right continuous bifunctor, as it is a left adjoint in both variables.

(2): Assume that $1$ is $\FP_2$. Every $S\in {}_A\S$ is dualizable, so the functor $-\otimes_A S$ is exact. Thus, to establish \Sref{setup2} it remains to prove the two natural isomorphisms herein. We only prove the second of these, i.e.~\mbox{$\Ext_{\C_0}(1,-\otimes_A S)\cong\Ext_{\C_A}(S^*,-)$} for $S \in {}_A\S$. The first one is proved similarly. To this end, we apply Lemma~\ref{HJ} to the adjunction \mbox{$S^*\mspace{-2mu}\otimes_1- \colon \C_0 \leftrightarrows \C_A \,\colon\! -\otimes_A \mspace{2mu}S$} from \Lref{finitedeflem}~\eqref{four}. The right adjoint $-\otimes_AS$ is clearly exact as $S$ is dualizable in ${}_A\C$. It remains to show that the left adjoint functor $S^*\otimes_1-$ leaves every short exact sequence $0 \to D \to E \to 1 \to 0$ (ending in $1$) exact. To see this, first note that the category $\C_0$ has \emph{enough $\otimes_1$-flats}, that is, for every object $X \in \C_0$ there exists an epimorphism $F \twoheadrightarrow X$ in $\C_0$ where $F$ is $\otimes_1$-flat.
Indeed, this follows from Stenstr{\"o}m~\cite[IV.6 Prop.~6.2]{Stenstrom} as $\C_0$ has coproducts and is generated by a set of $\otimes_1$-flat (even dualizable) objects. Consequently, $S^*$ has a $\otimes_1$-flat resolution $F_{\scriptscriptstyle\bullet} = \cdots \to F_1 \to F_0 \to 0$ in $\C_0$. Every short exact sequence $0 \to D \to E \to 1 \to 0$ in $\C_0$ induces a short exact sequence \mbox{$0 \to F_{\scriptscriptstyle\bullet} \mspace{-1mu}\otimes_1\mspace{-1mu} D \to F_{\scriptscriptstyle\bullet} \mspace{-1mu}\otimes_1\mspace{-1mu} E \to F_{\scriptscriptstyle\bullet}\to 0$} of chain complexes in $\C_0$ which, in turn, yields a long exact sequence in homology,
\begin{equation*}
  \cdots 
  \to H_1(F_{\scriptscriptstyle\bullet})    
  \to H_0(F_{\scriptscriptstyle\bullet}\mspace{-1mu}\otimes_1\mspace{-1mu}D)
  \to H_0(F_{\scriptscriptstyle\bullet}\mspace{-1mu}\otimes_1\mspace{-1mu}E)  
  \to H_0(F_{\scriptscriptstyle\bullet})  
  \to H_{-1}(F_{\scriptscriptstyle\bullet}\mspace{-1mu}\otimes_1\mspace{-1mu}D)  
  \to \cdots.
\end{equation*}
Evidently, $H_1(F_{\scriptscriptstyle\bullet})=0=H_{-1}(F_{\scriptscriptstyle\bullet}\mspace{-1mu}\otimes_1\mspace{-1mu}D)$. As the functor $-\otimes_1X$ is right exact we get $H_0(F_{\scriptscriptstyle\bullet}\otimes_1X) \cong S^*\otimes_1X$ for all $X \in \C_0$, and so \mbox{$0 \to S^* \mspace{-1mu}\otimes_1\mspace{-1mu} D \to S^* \mspace{-1mu}\otimes_1\mspace{-1mu} E \to S^* \to 0$} is exact, as desired.

(3) Immediate from Corollary~\ref{stronglimcor}.
\end{proof}

For closed symmetric monoidal Grothendieck categories we get the following.

\begin{cor}\label{cor-concretemodelexample} 
  Let $(\C_0,\otimes_1,[-,-],1)$ be a closed symmetric monoidal Grothendieck category where $1$ is finitely presented. Assume that $\C_0$ is generated by the set $\S$ of dualizable objects. Then the following hold:
  
  \begin{enumerate}

  \item Every semi-flat object in $\C$ belongs to $\dlim \S$.
  
  \item If $1$ is $\FP_2$, then the class of semi-flat objects in $\C$ is precisely $\dlim \S$.
  
  \item If $1$ is projective, then $\dlim \S$ is precisely the (tensor-)flat objects in $\C$.
  \end{enumerate}
\end{cor}

\begin{remark}\label{setup2rem} 
Consider the situation from \Tref{concretemodelexample}. If $1 \in \C_0$ is projective, then all objects in ${}_A\S$ are projective. Indeed, consider any $S \in {}_A\S$. By \Lref{finitedeflem}~\eqref{four} we have the adjunction \mbox{$S\otimes_1- \colon \C_0 \leftrightarrows {}_A\C \,\colon\! S^*\otimes_A-$}, and since the right adjoint functor $S^*\otimes_A-$ is exact, the left adjoint functor $S\otimes_1-$ preserves projective objects. Hence, if $1 \in \C_0$ is projective, then so is $S\otimes_1\mspace{-2mu}1 \cong S \in {}_A\S$. 
    
    Thus, in the case where $1 \in \C_0$ is projective, the cotorsion pair $(\P,\E)$ in ${}_A\C$ generated by ${}_A\S$ is the trivial cotorsion where $\P$ is the class of all projective objects and $\E={}_A\C$ (cf.~Definition~\ref{projacycdef}). Similarly for $\S_A$ and $\C_A$. 
\end{remark}

\section{Examples}
\label{Examples}

In this final section, we return to the examples from Example~\ref{examples-of-Setup1} and to the results from the literature mentioned in the Introduction.

\subsection{\texorpdfstring{$A$}{A}-Mod}
$\C_0=\Ab$ is a Grothendieck category generated by $1=\Z$, which is finitely presented and projective, and ${}_A\C$ is just $A$-Mod.
The condition in \Dref{finitedef} is equivalent to the existence of
a finite number of elements $f_i\in X^*$ and $x_i\in X$ such that $x=\sum_if_i(x)x_i$ for any $x\in X.$
By the Dual Basis Theorem \cite[Chap.~2.3]{Northcott}, this is precisely the finitely generated projective $R$-modules.
Also the finitely generated free modules are closed under extensions and contains $R\otimes_{\Z}\Z \cong R$,
so by \Tref{concretemodelexample}(3) we get
the original theorem of Lazard and Govorov:

\begin{cor}
  Over any ring, the flat modules are the direct limit closure of the finitely generated projective (or free) modules. \qed
\end{cor}

\subsection{\texorpdfstring{$A$}{A}-GrMod}
$\C_0=\mathbb{Z}\text{-GrMod}$ is a Grothendieck category where $1=\Z$ is finitely presented and projective. The category $\C_0$ is generated by the set $\S=\{\Sigma^i1\}_{i \in \Z}$, which is self-dual (that is, $\S^*=\S$) and consists of dualizable objects. Also note that ${}_A\C$ is just $A\textnormal{-GrMod}$. A graded $A$-module is \emph{finitely generated free} if it is a finite direct sums of shifts of $A$, and it is \emph{finitely generated projective} if it is a direct summand of a finitely generated free graded $A$-module. Arguments like the ones above show that the dualizable objects in ${}_A\C$ are precisely the finitely generated projective graded $A$-modules. Thus by \Tref{concretemodelexample}(3) we get the following version of Govorov-Lazard for graded modules (which does not seem to be available in the literature):

\begin{cor}
  \label{graded-modules}
  Over any $\mathbb{Z}$-graded ring, the flat graded modules are the direct limit closure of the finitely generated projective (or free) graded modules. \qed
\end{cor}

\subsection{\texorpdfstring{$A$}{A}-DGMod}
\label{subsectionDGMOD}
$\C_0=\Ch(\Ab)$, the category of chain complexes of abelian groups, is a Grothendieck category where $1$ is the complex with $\Z$ concentrated in degree~$0$. Note that $1$ is finitely presented (but not projective!), as $\C_0(1,-)\cong Z_0(-)$ is the $0^\mathrm{th}$ cycle functor which preserves direct limits. The category $\C_0$ is generated by the set $\S=\{\Sigma^iM(\mathrm{Id}_1)\}_{i \in \Z}$ (where $M(\mathrm{Id}_1)$ is the mapping cone of the identity morphism on $1$), which is self-dual (i.e.~$\S^*=\S$) and consists of dualizable objects.

A monoid $A$ in $\C_0=\Ch(\Ab)$ is a differential graded algebra and
${}_A\C$ is the category $A\text{-DGMod}$ of differential graded left $A$-modules.
DG-modules are thus covered by \Sref{setup3}.
Clearly any shift of $A$ is dualizable, so by \Lref{finiteext} any finite extension of shifts of $A$ will be dualizable, and we call such modules
\emph{finitely generated semi-free}. Direct summands of those are called \emph{finitely generated semi-projectives}. I do not know if the finitely generated semi-projective DG-modules constitute all dualizable objects in ${}_A\C = A\text{-DGMod}$ for a general DGA. Nevertheless, it is not hard to check that both of the above mentioned classes are self-dual and closed under extensions.
They also contain $A\otimes_{\Z}\S$ and thus satisfy \Sref{setup1} by \Tref{concretemodelexample}(1). 

Actually, $1=\Z$ is not just finitely presented but even $\FP_2$, so from \Tref{concretemodelexample}(2) we conclude that the direct limit closure of the finitely generated semi-free/semi-projective DG-modules is precisely the class of semi-flat objects in ${}_A\C = A\text{-DGMod}$ in the abstract sense of \Dref{projacycdef}. Before we go further into this, let's see that our abstract notions of semi-projective, acyclic and semi-flat objects from \Dref{projacycdef} agree with the usual ones. These notions originate in the treatise \cite{avramov} by Avramov, Foxby, and Halperin, where several equivalent conditions are given.

\begin{defn} 
\label{DG-notions}
Let $A$ be any DGA and let ${}_A\C=A\text{-DGMod}$.
  \begin{itemize}
  \item A DG-module is called \emph{acyclic} (or \emph{exact}) if it has trivial homology.
  \item A DG-module, $P$, is called \emph{semi-projective} (or \emph{DG-projective}) 
        if ${}_A\C(P,\psi)$ is epi, whenever $\psi$ is epi and $\ker\psi$ has trivial homology (in other words, $\psi$ is a surjective quasi-isomorphism).
  \item A DG-module, $M,$ is called \emph{semi-flat} (or \emph{DG-flat}) if $-\otimes_A M$ is exact and preserves acyclicity
        (i.e. $E\otimes_A M$ has trivial homology whenever $E$ has).    
  \end{itemize}
\end{defn}

First we notice that:

\begin{lemma}\label{dgproj}
 A DG-module $P$ is DG-projective iff $\Ext^1_{{}_A\C}(P,E)=0$ whenever $E$ is a DG-module with trivial homology.
\end{lemma}
\begin{proof}
If $\Ext^1_{{}_A\C}(P,E)=0$ and $$\xymatrix{
0\ar[r] &E\ar[r]&A\ar[r]^\phi&B\ar[r]&0
}$$
is an exact sequence, then clearly ${}_A\C(P,\phi)$ is epi.
On the other hand, if $$\xymatrix{
0\ar[r] &E\ar[r]&X\ar[r]^\phi&P\ar[r]&0
}$$
is exact and ${}_A\C(P,\phi)$ is epi, then the sequence split, so $\Ext^1_{{}_A\C}(P,E)=0.$
\end{proof}

Next we see that:

\begin{lemma}\label{DGhomology}
Let $A$ be a DGA. For any $N\in{}_A\C$ we have
 $\Ext^1_{{}_A\C}(\Sigma A,N)=H_0(N).$  
\end{lemma}
\begin{proof}
  To compute this, we use the short exact sequence
$$
\xymatrix{
0\ar[r]& A\ar[r]&M(\Id_A)\ar[r]&\Sigma A\ar[r]&0
}
$$
where $M(\Id_A)$ is the mapping cone of \smash{$\xymatrix{A\ar[r]^{\Id_A}&A}$}.
Since $M(\Id_A)$ is projective we have $\Ext^1_{{}_A\C}(M(\Id_A),N)=0,$ so we get an exact sequence
$$
\xymatrix{
{}_A\C(M(\Id_A),N)\ar[r]&{}_A\C(A,N)\ar[r]&\Ext^1_{{}_A\C}(\Sigma A,N)\ar[r]&0
}
$$
Straightforward calculations show that this sequence is isomorphic to
\begin{equation*}
\xymatrix{
N_{1}\ar[r]^-{\partial_{1}^N}&Z_0(N)\ar[r]&\Ext^1_{{}_A\C}(\Sigma A,N)\ar[r]&0
}
\end{equation*}
where $N_{1}$ is the degree $1$ part of $N$ and $\partial_{1}^N$ is the differential. Thus we get the desired~iso\-morphim $\Ext^1_{{}_A\C}(\Sigma A,N)\cong H_0(N)$.
\end{proof}

Together we have the following.

\begin{thm}\label{thmDG}
Let $A$ be any DGA and let $\S$ be the class of finitely generated semi-free/semi-projective DG $A$-modules (see \ref{subsectionDGMOD}). The abstract notions of semi-projec\-ti\-vi\-ty, acyclicity, and semi-flatness from \Dref{projacycdef} agree with the corresponding DG notions from \Dref{DG-notions}. In the category of DG $A$-modules, the~co\-tor\-sion pair generated by $\S$ is complete and it is given by
$$(\text{DG-projective DG-modules}, \text{exact DG-modules})\,.$$
The direct limit closure of $\S$ is the class of semi-flat (or DG-flat) DG-modules.
\end{thm}

\begin{proof}
Let $\P$ be the class of DG-projective DG-modules, and $\E$ the class of exact DG-modules (i.e.~with trivial homology).
  From \Lref{DGhomology} (and by using shift $\Sigma$) we have $\S^\perp\subset\E,$  and from \Lref{dgproj} we have $\P={}^\perp\E.$
  Now since $\S\subset\P$ we have  $\E\subset({}^\perp\E)^\perp=\P^\perp\subset\S^\perp$, and hence  $({}^\perp(\S^\perp),\S^\perp)=(\P,\E)$. This shows that the abstract notions of semi-projec\-ti\-vi\-ty and acyclicity agree with the corresponding DG notions. Completeness of the cotorsion pair $(\P,\E)$ follows from \Pref{completecotorsion}, as already mentioned in Definition~\ref{projacycdef}. It remains to see that the abstract notion of semi-flatness agrees with the corresponding DG notion. It must be shown that if $M$ is a left DG $A$-module that satisfies $E \otimes_A M \in 1^\perp$, i.e.~$\Ext^1_{\Ch(\Ab)}(\Z,E \otimes_A M)=0$, for all acyclic right DG $A$-modules $E$, then $E \otimes_A M$ has trivial homology for all such $E$'s. However, by \Lref{DGhomology} we have $\Ext^1_{\Ch(\Ab)}(\Z,E \otimes_A M)=H_{-1}(E \otimes_A M)$, so the conclusion follows as $- \otimes_A M$ preserves shifts. The last statement in the theorem follows from \Tref{concretemodelexample}(2); cf.~the discussion in \ref{subsectionDGMOD}.
\end{proof}

\begin{remark}
The cotorsion pair is well-known. It is one of the cotorsion pairs corresponding (via Hovey \cite[Thm 2.2]{MR1938704})
to the standard projetive model structure on $A$-DGMod (see for instance Keller \cite[Thm 3.2]{keller06}).
That every \mbox{$S\in\dlim\S$} is semi-flat follows directly from results in \cite{avramov},
where it is proved that any semi-projective is semi-flat and that the semi-flats are closed under direct limits.
That every semi-flat can be realized as a direct limit of finitely genereated semi-free/projectives is, to the best of my knowledge, new. 
\end{remark}

\subsection{Ch\texorpdfstring{$(A)$}{(A)}}
In the case of complexes over a ring $A$ a direct calculation using the dual basis theorem component-wise,
shows that the dualizable objects in $\Ch(A)$ are precisely the perfect complexes. From above we thus have:

\begin{cor}
  Let $A$ be any ring and let $\S$ be the class of perfect $A$-complexes. In the category $\Ch(A)$, the cotorsion pair generated by $\S$ is complete and it is given by (semi-projective complexes, acyclic complexes). The direct limit closure of $\S$ is the class of semi-flat complexes. \qed
\end{cor}

\begin{remark}
  This cotorsion pair has already been studied for instance in \cite{rozas99} where
2.3.5 and 2.3.6 proves it is a cotorsion pair, and 2.3.25 that it is complete (with slightly different notation).
It is not mentioned, however, that it is generated by a set.
As already mentioned in the Introduction, the direct limit closure has in this case been worked out in \cite{holm14}.
\end{remark}

\subsection{QCoh(X)}
\label{subsection:QCoh}
Let $X$ be any scheme and let $\QCoh(X)$ be the category of quasi-coherent sheaves (of $\mathscr{O}_X$-modules) on $X$. This is an abelian and a symmetric~mo\-noi\-dal subcategory of $\Mod(X)$  (the category of all sheaves on $X$), see \cite[II Prop.~5.7]{Hartshorne} and \cite[Tag 01CE]{stacks-project}. It is also a Grothendieck category, indeed, most of the relevant properties of $\QCoh(X)$ go back to Grothendieck \cite{MR0102537,EGAI}; the existence of a generator is an unpublished result by Gabber (1999), see \cite[Tag 077K]{stacks-project} and Enochs and Estrada \cite{MR2139915} for a proof. The  symmetric monoidal category $\QCoh(X)$ is also closed: as explained in \cite[3.7]{MR2419383}, the internal hom in $\QCoh(X)$ is constructed from that in $\Mod(X)$ composed with the quasi-coherator (the right adjoint of the inclusion $\QCoh(X) \to \Mod(X)$), which always exists \cite[Tag 077P]{stacks-project}.

The \emph{dualizable} objects in $\QCoh(X)$ are also studied in Brandenburg \cite[Def. 4.7.1 and Rem. 4.7.2]{brandenburg14}, and
\cite[Prop.~4.7.5]{brandenburg14} shows that they are exactly the locally free sheaves of finite rank. Recall from Sch\"appi~\cite[Def.~6.1.1]{Schappi} (see also \cite[Def.~2.2.7]{brandenburg14}) that a scheme $X$ is said to have the \emph{strong resolution property} if $\QCoh(X)$ s generated by  locally free sheaves of finite rank. This is the case if $X$ is e.g.~a separated noetherian scheme with a family of ample line bundles; see Hovey \cite[Prop.~2.3]{MR1814077} and Krause \cite[Exa.~4.8]{MR2157133}.

An object $M \in \QCoh(X)$ is \emph{semi-flat} if it is so in the sense of \Dref{projacycdef}, that is, if the functor 
\mbox{$-\otimes_{\mathscr{O}_X} M$} is exact and 
\smash{$\Ext^1_{\QCoh(X)}(\mathscr{O}_X,N \otimes_{\mathscr{O}_X} M)=0$} holds for all $N \in \QCoh(X)$ for which \smash{$\Ext^1_{\QCoh(X)}(S,N)=0$} for all locally free sheaves $S$ of finite rank. Now, from Corollary~\ref{cor-concretemodelexample} we get:

\begin{prop}\label{sheaves}
  Let $(X,\mathscr{O}_X)$ be a scheme with the strong resolution property.
\begin{enumerate}
\item If $\mathscr{O}_X$ is $\FP_1$, then every semi-flat object in $\QCoh(X)$ is a direct limit of locally free sheaves of finite rank.

\item If $\mathscr{O}_X$ is $\FP_2$ then, conversely, every direct limit in $\QCoh(X)$ of locally free sheaves of finite rank is semi-flat. \qed
\end{enumerate} 
\end{prop}

\begin{remark}
  \label{Qcoh-vs-Mod}
  It follows from \cite[II Thm.~7.18]{MR0222093} that if $X$ is locally noetherian, then every injective object in $\QCoh(X)$ is also injective in $\Mod(X)$. Thus, in this case one has $\Ext^i_{\QCoh(X)}(M,N) \cong \Ext^i_{\Mod(X)}(M,N)$ for all $M,N \in \QCoh(X)$.
\end{remark}

\begin{thm}
  \label{sheaves-noetherian}
  Let $X$ be a noetherian scheme with the strong resolution property. In the category $\QCoh(X)$, the direct limit closure of the locally free sheaves of finite rank is precisely the class of semi-flat sheaves.
\end{thm}

\begin{proof}
  As $X$ is, in particular, a locally noetherian scheme, Remark~\ref{Qcoh-vs-Mod} and \cite[III Prop.~6.3(c)]{Hartshorne} shows that \smash{$\Ext^i_{\QCoh(X)}(\mathscr{O}_X,-) \cong H^i(X,-)$} for all $i \geqslant 0$. If we view $H^i(X,-)$ as a functor $\Mod(X) \to \Ab$, then it preserves direct limits by \cite[III Prop.~2.9]{Hartshorne} as $X$ is a noetherian scheme (see also \cite[III 3.1.1]{Hartshorne}). But then $H^i(X,-)$ also preserves direct limits as a functor $\QCoh(X) \to \Ab$ since colimits in $\QCoh(X)$ are just computed in $\Mod(X)$, see \cite[Tag 01LA]{stacks-project}. We conclude that $\mathscr{O}_X$ is both $\FP_1$ and $\FP_2$ and the desired conclusion follows from \Pref{sheaves}.
\end{proof}

This is not the first Lazard-like theorem for quasi-coherent sheaves. The usual notion of flatness is \emph{locally flat}, which means that the stalks are flat. Such sheaves are tensor-flat, and the converse holds if the scheme is quasi-separated \cite[Lem.~4.6.2]{brandenburg14}.

In \cite[(5.4)]{boevey94} Crawley-Boevey proves that $\dlim\S$ is precisely the locally flat sheaves if $X$ is a non-singular irreducible curve or surface over a field $k$.
 
In \cite[2.2.4]{brandenburg14} Brandenburg proves that if $X$ has the strong resolution property and $M$ is locally flat and
Spec$($Sym$(M))$ is affine, then $M\in\dlim\S.$

Thus for a scheme with the strong resolution property we have the relations:
$$
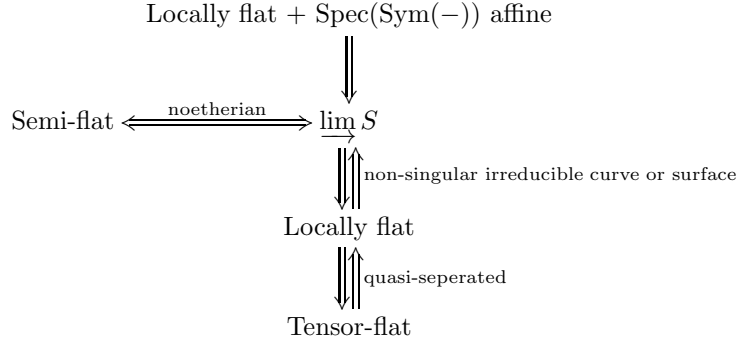
\xymatrix@C=0.5pc{
&\text{Locally flat + Spec$($Sym$(-))$ affine}\ar@{=>}[d]  \\
\text{Semi-flat} \ar@{<=>}[r]^-{\text{noetherian}}
&\dlim S\ar@<-.6ex>@{=>}[d]\ar@<.6ex>@{<=}[d]^{\text{non-singular irreducible curve or surface}} \\ 
&\text{Locally flat} \ar@<-.6ex>@{=>}[d]\ar@<.6ex>@{<=}[d]^{\text{quasi-seperated}} \\
&\text{Tensor-flat}
}
$$

It would be interesting to get a concrete description of the
semi-flat, the acyclic and the semi-projective objects in $\QCoh(X)$.

Some work has been done in this direction.
In Enochs, Estrada, and Garc{\'{\i}}a-Rozas \cite[3.1]{estrada08} we see that the semi-projective objects are locally projective,
and in \cite[4.2]{estrada08} we see that they are precisely the locally projective sheaves in the special case of $P_1(k)$ (the projective line over an algebraically closed field $k$).
In this case a concrete computational description of the (abstract) acyclic objects are given. I am not aware of anybody explicitly studying semi-flat sheaves.

\subsection{Additive functors}
\label{subsection:additive-functors}
Following \cite{oberst70}, let \mbox{$\C_0=\Ab$}, let $\X$ be a small preadditive category, let $\X^\mathrm{op}$ be the dual category, 
let $\C_L=[\X,\Ab]$ and $\C_R=[\X^\mathrm{op},\Ab]$ be the categories of additive functors, and let $\S$ be the class of finite direct sums of representable functors (recall that the \emph{representable} functors in $\C_L$ and $\C_R$ are the functors $\X(x,-)$ and $\X(-,x)$ where \mbox{$x\in\X$}). We
define $\X(-,x)^*=\X(x,-)$ and vice versa. As in \cite{oberst70} one can define a tensorproduct
$$\otimes_\X\colon [\X^\mathrm{op},\Ab]\times [\X,\Ab]\longrightarrow\Ab.$$
We claim that these data satisfy \Sref{setup1}: The categories $\C_L$ and $\C_R$ are Grothen\-dieck and generated by $\S$; see~\cite[Lem. 2.4]{oberst70}. Note that $\S$ is small as $\X$ is small. Furthermore, $\S$ is closed under extensions; indeed
the objects in $\S$ are projective (in fact, every projective object is a direct summand of an object from $\S$), hence any extension is a direct sum. If $\X$ is additive, then $\X(-,x)\oplus\X(-,y)\cong\X(-,x\oplus y)$, so in this case $\S$ is just the class of representable functors (finite direct sums are not needed). Further, as in \cite{oberst70} the tensor product is such that for any $F\in\C_L$ and $G\in\C_R$ we have
\begin{equation*}
\X(-,x)\otimes_\X F\cong Fx \quad \text{ and } \quad G\otimes_\X\X(x,-)\cong Gx\;,
\end{equation*}
which by the Yoneda lemma, and the fact that  $\Ab(1,-)$ is the identity gives the required isomorphisms from \Sref{setup1}. As the functors $\X(-,x)\otimes_\X?$ and $?\otimes_\X\X(x,-)$ are nothing but evaluation at $x$, they are exact. Finally, as $1=\Z\in\Ab$ is projective, \Coref{stronglimcor} gives a new proof of \cite[Thm 3.2]{oberst70}:

\begin{cor}
Let $\X$ be an additive category, and let $\S$ be the finitely generated projective functors or the representable functors in $[\X,\Ab]$
(or the direct sums of representable functors if $\X$ is only preadditive).
A functor $F$ is flat iff $F\in\dlim\S.$
\end{cor}

\section*{Acknowledgements}

I would like to thank my advisor Henrik Holm for suggesting the topic and~for his guidance. I thank Sergio Estrada for discussing and providing references for the case of quasi-coherent sheaves. I thank Luchezar L. Avramov for making available his unpublished manuscript on differential graded homological algebra, joint with 
Hans-Bj{\o}rn Foxby and Stephen Halperin, which has served as a key source of inspiration. Finally, it is a pleasure to thank the anonymous referee for his/her thorough and insightful comments that greatly improved the manuscript and strengthened \Tref{concretemodelexample} (and its applications) significantly.

\bibliographystyle{amsplain}
\bibliography{bibliography}

\providecommand{\bysame}{\leavevmode\hbox to3em{\hrulefill}\thinspace}
\providecommand{\MR}{\relax\ifhmode\unskip\space\fi MR }
\providecommand{\MRhref}[2]{%
  \href{http://www.ams.org/mathscinet-getitem?mr=#1}{#2}
}
\providecommand{\href}[2]{#2}
\begin{thebibliography}{10}

\bibitem{MR2419383}
L.~Alonso~Tarr{\'\i}o, A.~Jerem{\'\i}as~L\'opez, M.~P{\'e}rez~Rodr{\'\i}guez,
  and M.~J. Vale~Gonsalves, \emph{The derived category of quasi-coherent
  sheaves and axiomatic stable homotopy}, Adv. Math. \textbf{218} (2008),
  no.~4, 1224--1252.

\bibitem{avramov}
L.~Avramov, H.~Foxby, and S.~Halperin, \emph{Differential graded homological
  algebra}, 1994–2014, preprint.

\bibitem{brandenburg14}
M.~Brandenburg, \emph{Tensor categorical foundations of algebraic geometry},
  2014, preprint, arXiv:1410.1716.

\bibitem{breit79}
S.~Breitsprecher, \emph{Lokal endlich pr{\"a}sentierbare
  {G}rothendieck-{K}ategorien}, Mitt. Math. Sem. Giessen Heft \textbf{85}
  (1970), 1--25.

\bibitem{holm14}
L.~W. Christensen and H.~Holm, \emph{The direct limit closure of perfect
  complexes}, J. Pure Appl. Algebra \textbf{219} (2015), no.~3, 449--463.

\bibitem{boevey94}
W.~Crawley-Boevey, \emph{Locally finitely presented additive categories}, Comm.
  Algebra \textbf{22} (1994), no.~5, 1641--1674.

\bibitem{MR2139915}
E.~E. Enochs and S.~Estrada, \emph{Relative homological algebra in the category
  of quasi-coherent sheaves}, Adv. Math. \textbf{194} (2005), no.~2, 284--295.

\bibitem{estrada08}
E.~E. Enochs, S.~Estrada, and J.~R. Garc{\'{\i}}a-Rozas, \emph{Locally
  projective monoidal model structure for complexes of quasi-coherent sheaves
  on {$\bold P\sp 1(k)$}}, J. Lond. Math. Soc. (2) \textbf{77} (2008), no.~1,
  253--269.

\bibitem{rozas99}
J.~R. Garc{\'{\i}}a~Rozas, \emph{Covers and envelopes in the category of
  complexes of modules}, Chapman \& Hall/CRC Res. Notes Math., vol. 407,
  Chapman \& Hall/CRC, Boca Raton, FL, 1999.

\bibitem{trlifaj06}
R.~G{\"o}bel and J.~Trlifaj, \emph{Approximations and endomorphism algebras of
  modules}, de Gruyter Exp. Math., vol.~41, Walter de Gruyter GmbH \& Co. KG,
  Berlin, 2006.

\bibitem{gov65}
V.~E. Govorov, \emph{On flat modules}, Sibirsk. Mat. \v Z. \textbf{6} (1965),
  300--304.

\bibitem{MR0102537}
A.~Grothendieck, \emph{Sur quelques points d'alg\`ebre homologique}, T\^ohoku
  Math. J. (2) \textbf{9} (1957), 119--221.

\bibitem{EGAI}
A.~Grothendieck and J.~A. Dieudonn\'e, \emph{El\'ements de g\'eom\'etrie
  alg\'ebrique. {I}}, Grundlehren Math. Wiss., vol. 166, Springer-Verlag,
  Berlin, 1971.

\bibitem{MR0222093}
R.~Hartshorne, \emph{Residues and duality}, Lecture notes of a seminar on the
  work of A. Grothendieck, given at Harvard 1963/64. With an appendix by P.
  Deligne. Lecture Notes in Math., No. 20, Springer-Verlag, Berlin-New York,
  1966.

\bibitem{Hartshorne}
\bysame, \emph{Algebraic geometry}, Springer-Verlag, New York-Heidelberg, 1977,
  Graduate Texts in Mathematics, No. 52.

\bibitem{HJ}
H.~Holm and P.~J{\o}rgensen, \emph{Cotorsion pairs in categories of quiver
  representations}, 2016, Kyoto J. Math. (to appear), arXiv:1604.01517v2.

\bibitem{MR1814077}
M.~Hovey, \emph{Model category structures on chain complexes of sheaves},
  Trans. Amer. Math. Soc. \textbf{353} (2001), no.~6, 2441--2457.

\bibitem{MR1938704}
\bysame, \emph{Cotorsion pairs, model category structures, and representation
  theory}, Math. Z. \textbf{241} (2002), no.~3, 553--592. \MR{1938704}

\bibitem{HPS97}
M.~Hovey, J.~H. Palmieri, and N.~P. Strickland, \emph{Axiomatic stable homotopy
  theory}, Mem. Amer. Math. Soc. \textbf{128} (1997), no.~610, x+114.

\bibitem{keller06}
B.~Keller, \emph{On differential graded categories}, International {C}ongress
  of {M}athematicians. {V}ol. {II}, Eur. Math. Soc., Z\"urich, 2006,
  pp.~151--190.

\bibitem{MR2157133}
H.~Krause, \emph{The stable derived category of a {N}oetherian scheme}, Compos.
  Math. \textbf{141} (2005), no.~5, 1128--1162.

\bibitem{lazard68}
D.~Lazard, \emph{Autour de la platitude}, Bull. Soc. Math. France \textbf{97}
  (1969), 81--128.

\bibitem{equiva}
L.~G. Lewis, J.~P. May, and M.~Steinberger, \emph{Equivariant stable homotopy
  theory}, Lecture Notes in Math., vol. 1213, Springer-Verlag, Berlin, 1986,
  With contributions by J. E. McClure.

\bibitem{Northcott}
D.~G. Northcott, \emph{A first course of homological algebra}, Cambridge
  University Press, London, 1973.

\bibitem{oberst70}
U.~Oberst and H.~R{\"o}hrl, \emph{Flat and coherent functors}, J. Algebra
  \textbf{14} (1970), 91--105.

\bibitem{pareigis86}
B.~Pareigis, \emph{Non-additive ring and module theory. {I}. {G}eneral theory
  of monoids}, Publ. Math. Debrecen \textbf{24} (1977), no.~1-2, 189--204.

\bibitem{salce79}
L.~Salce, \emph{Cotorsion theories for abelian groups}, Symposia {M}athematica,
  {V}ol. {XXIII} ({C}onf. {A}belian {G}roups and their {R}elationship to the
  {T}heory of {M}odules, {INDAM}, {R}ome, 1977), Academic Press, London-New
  York, 1979, pp.~11--32.

\bibitem{stovicek11}
M.~Saor{\'{\i}}n and J.~{\v{S}}{\v{t}}ov{\'{\i}}{\v{c}}ek, \emph{On exact
  categories and applications to triangulated adjoints and model structures},
  Adv. Math. \textbf{228} (2011), no.~2, 968--1007.

\bibitem{Schappi}
D.~Sch\"appi, \emph{A characterization of categories of coherent sheaves of
  certain algebraic stacks}, 2012, to appear in J. Pure Appl. Algebra,
  arXiv:1206.2764.

\bibitem{stacks-project}
The {Stacks Project Authors}, \emph{Stacks project},
  \url{http://stacks.math.columbia.edu}, 2017.

\bibitem{Stenstrom}
B.~Stenstr\"om, \emph{Rings of quotients}, Grundlehren Math. Wiss., vol. 217,
  Springer-Verlag, New York-Heidelberg, 1975.

\bibitem{stovicek13}
J.~{\v{S}}{\v{t}}ov{\'{\i}}{\v{c}}ek, \emph{Deconstructibility and the {H}ill
  lemma in {G}rothendieck categories}, Forum Math. \textbf{25} (2013), no.~1,
  193--219.

\end{thebibliography}

\end{document}